\documentclass[letterpaper, 10 pt, conference]{ieeeconf}  
\usepackage{amsmath}
\usepackage{amssymb}
\usepackage{enumerate}
\usepackage{algorithm}
\usepackage{tikz}
\usepackage{tcolorbox}
\usepackage{url} 
\usepackage[noend]{algpseudocode}
\usepackage{graphicx}
\usepackage{pgfplots}
\usepackage{pgfplotstable}
\usepackage{mathtools}
\usepackage[noadjust]{cite}
\usepgfplotslibrary{fillbetween} 
\usepackage{subcaption}
\usepackage[colorinlistoftodos,prependcaption]{todonotes}
\IEEEoverridecommandlockouts                              
\algrenewcommand\textproc{\textsc}

\pgfplotstablesort[sort key=time]{\mptglpk}{data/MPTglpk.dat}
\pgfplotstablesort[sort key=time]{\POP}{data/POP_2018a_CPLEX.dat}
\pgfplotstablesort[sort key=time]{\cholrank}{data/result_cholrankcheck.dat}

\title{A High-Performant Multi-Parametric Quadratic Programming Solver}

\author{Daniel Arnström, Daniel Axehill%
\thanks{D. Arnstr\"om is with the Division of Systems and Control, Depratment of Information Technology, Uppsala University, Sweden 
{\tt\small daniel.arnstrom@it.uu.se}}%
\thanks{D. Axehill rs with the Division of Automatic Control, Link\"oping University,
  Sweden {\tt\small daniel.axehill@liu.se}}%
}

\begin{document}
\renewcommand{\baselinestretch}{1.0}

\definecolor{set19c1}{HTML}{E41A1C}
\definecolor{set19c2}{HTML}{377EB8}
\definecolor{set19c3}{HTML}{4DAF4A}
\definecolor{set19c4}{HTML}{984EA3}
\definecolor{set19c5}{HTML}{FF7F00}
\definecolor{set19c6}{HTML}{FFFF33}
\definecolor{set19c7}{HTML}{A65628}
\definecolor{set19c8}{HTML}{F781BF}
\definecolor{set19c9}{HTML}{999999}

\maketitle
\thispagestyle{empty}
\pagestyle{empty}
\newtheorem{proposition}{Proposition}
\newtheorem{lemma}{Lemma}
\newtheorem{corollary}{Corollary}
\newtheorem{remark}{Remark}
\newtheorem{theorem}{Theorem}
\newtheorem{definition}{Definition}
\newtheorem{assumption}{Assumption}
\newtheorem{example}{Example}
\newtheorem{problem}{Problem}

\begin{abstract}
    We propose a combinatorial method for computing explicit solutions to multi-parametric quadratic programs, which can be used to compute explicit control laws for linear model predictive control. In contrast to classical methods, which are based on geometrical adjacency, the proposed method is based on combinatorial adjacency. After introducing the notion of combinatorial adjacency, we show that the explicit solution forms a connected graph in terms of it. We then leverage this connectedness to propose an algorithm that computes the explicit solution. The purely combinatorial nature of the algorithm leads to computational advantages since it enables demanding geometrical operations (such as computing facets of polytopes) to be avoided. Compared with classical combinatorial methods, the proposed method requires fewer combinations to be considered by exploiting combinatorial connectedness. We show that an implementation of the proposed method can yield a speedup of about two orders of magnitude compared with state-of-the-art software packages such as \texttt{MPT} and \texttt{POP}. 
\end{abstract}

\section{Introduction}
In Model Predictive Control (MPC), a control action is determined at each time step by solving an optimization problem \cite{rawlings2017model}.
When the dynamics of the system to be controlled is linear, the optimization problems in question can be cast as instances of a multi-parametric quadratic program (mpQP) of the form  
\begin{equation}
  \begin{aligned}
      \label{eq:rmpc-mpqp}
	   &\underset{x}{\text{minimize}}&& \frac{1}{2}x^T H x + f(\theta)^T x \\
	   &\text{subject to} && A x \leq  b(\theta), \\
  \end{aligned}
\end{equation}
where the decision variable $x \in \mathbb{R}^n$ is related to the control action, and the parameter $\theta \in \Theta_0 \subseteq \mathbb{R}^p$ is related to setpoints and the system state. The parameter set $\Theta_0$ is assumed to be a polyhedron. For a given linear (and time-invariant) MPC application, the Hessian $H \succ 0$ and the constraint matrix $A \in \mathbb{R}^{m \times n}$ are \textit{constant}. Moreover,  both the linear cost $f: \mathbb{R}^p \to \mathbb{R}^n$ and the constraint offset $b:\mathbb{R}^p \to \mathbb{R}^m$ are affine functions of $\theta$ \cite{borrelli2017predictive}.
The particular structure of \eqref{eq:rmpc-mpqp} allows for a closed-form solution $x^*(\theta)$ that is piecewise affine over polyhedral regions. This closed-form, or \textit{explicit}, solution is used in \textit{explicit MPC}, where the control law is implemented as a simple lookup table \cite{bemporad2002explicit}.  

Albeit straightforward to theoretically derive the explicit solution, it is not as straightforward to compute the corresponding polyhedral regions efficiently and reliably. As a result, several methods for computing the explicit solution have been developed, which generally fall into two categories: geometrical \cite{bemporad2002explicit,tondel2003algorithm,grieder2004computation,jones2006multiparametric} and combinatorial \cite{gupta2011novel,feller2013improved,oberdieck2017explicit,ahmadi2018combinatorial,herceg2015enumeration}; an alternative approach for mpQPs that specifically originate from MPC, which is based on dynamic programming, has been proposed in \cite{mitze2020dynamic}. State-of-the-art software packages that can compute the explicit solution to \eqref{eq:rmpc-mpqp} include \texttt{MPT}~\cite{herceg2013mpt}, \texttt{POP}~\cite{oberdieck2016pop}, and the \texttt{Hybrid Toolbox}~\cite{hyb2004bemporad}.

The main contribution of this paper is a combinatorial method that efficiently computes the explicit solution of~\eqref{eq:rmpc-mpqp}. The method is based on exploring a connected graph, similar to \cite{oberdieck2017explicit} and \cite{ahmadi2018combinatorial}, to tame the combinatorial nature of computing the explicit solution. In contrast to \cite{oberdieck2017explicit}, the method does not rely on any geometrical operations such as computing the facets of polytopes, which makes the resulting method more efficient and reliable. In contrast to \cite{ahmadi2018combinatorial}, the proposed method handles degeneracies in a more straightforward manner; the proposed method does not, for example, need to explicitly check if constraints are weakly active/inactive. The method is also related to the complexity-certification method in \cite{arnstrom2022unifying}, which produces the explicit solution as a byproduct.

Concretely the main contributions of the paper are:
\begin{itemize}
    \item[(i)] Proving that the explicit solution to an mpQP form a connected graph in a combinatorial sense (Theorem \ref{th:main}).
    \item[(ii)] A combinatorial mpQP method that builds on exploring combinatorial adjacent active sets (Algorithm \ref{alg:updown}).
    \item[(iii)] An efficient implementation of the proposed method that is often several orders of magnitude faster than state-of-the-art software (Section \ref{sec:result}). 
\end{itemize}
The rest of the paper is organized as follows: In Section~\ref{sec:prelim} we describe how a multi-parametric least-distance problem (mpLDP) can be consider instead of the mpQP in \eqref{eq:rmpc-mpqp}. We then derive the explicit solution to this mpLDP and formalize a combinatorial problem for computing it. The section ends with a brief review of existing methods for computing the explicit solution. In Section~\ref{sec:algorithm} we introduce the concept of geometrical and combinatorial adjacency of active sets, and show that any pair of optimal active sets are connected by a sequence combinatorially adjacent acitve sets. We then leverage this connectedness to propose an algorithm that efficiently computes the explicit solution. In Section~\ref{sec:result} we show that an implementation of the proposed algorithm is about two orders of magnitude faster than the state-of-the-art mpQP solvers implemented in \texttt{MPT} \cite{herceg2013mpt} and \texttt{POP} \cite{oberdieck2016pop}.
\subsubsection*{Notation}
Subscript denotes indexing of the element/rows of vectors/matrices. For example, $v_i$ denotes the $i$th element of the vector $v$, and $M_{\mathcal{I}}$ denotes a submatrix of the matrix $M$ that is indexed by the set $\mathcal{I}$.  The complement of an index set $\mathcal{I}$ is denoted $\bar{\mathcal{I}}$ and its cardinality is denoted $|\mathcal{I}|$. The set of all integers between $1$ and $m$ is denoted $\mathbb{Z}_{1:m}$.
\section{Preliminaries}
\label{sec:prelim}
In this section we first transform \eqref{eq:rmpc-mpqp} into a multi-parametric least-distance problem (mpLDP), which simplifies the exposition,  and also improves computational aspects of the proposed algorithm. We then state the KKT conditions for this mpLDP and use them to characterize the explicit solution. Finally, we formalize the problem of computing the explicit solution (Problem \ref{prob:main}), and give a brief overview of existing methods for solving it. 
\subsection{Equivalent least-distance problem}

To simplify notation and reduce computations in the proposed algorithm, we first transform the mpQP in \eqref{eq:rmpc-mpqp} into the equivalent multi-parametric least-distance problem (mpLDP) of the form
\begin{equation}
    \label{eq:mpLDP}
  \begin{aligned}
	&\underset{u}{\text{minimize}}&&\frac{1}{2}\|u\|^2_2\\
	&\text{subject to} && M u \leq d(\theta),
  \end{aligned}
\end{equation}
by using the transformation $u = R(x+R^{-T}f(\theta))$, where $R$ is an upper Cholesky factor of $H$; the problem data in \eqref{eq:mpLDP} is, accordingly, defined as 
\begin{equation}
  \label{eq:aux-def}
  M \triangleq A R^{-1}, \qquad d(\theta)\triangleq b(\theta)+M R^{-T} f(\theta).
\end{equation}
The solution $x^*(\theta)$ to \eqref{eq:rmpc-mpqp} can be retrieved from the solution $u^*(\theta)$ to \eqref{eq:mpLDP} as 
\begin{equation}
    \label{eq:sol-transform}
    x^*(\theta) = R^{-1} \left(u^*(\theta)-R^{-T}f(\theta)\right).
\end{equation}

Importantly, we have that the affine structure of the constraint offset is retained, which we formalize in the following lemma.
\begin{lemma}[Affine offset]
    The offset $d: \mathbb{R}^p \to \mathbb{R}^m$ is an affine function of $\theta$.  
\end{lemma}
\begin{proof}
    From \eqref{eq:aux-def}, the offset $d$ is a linear transformation of $b$ and $f$, which are both affine functions of $\theta$. Since affine functions are preserved under linear transformations, $d$ is also an affine function of $\theta$.  
\end{proof}

\begin{remark}[Relating mpLDP and mpQP]
    Since $d(\theta)$ is affine, the LDP in \eqref{eq:mpLDP} is a special case of an mpQP of the form \eqref{eq:rmpc-mpqp}. Hence, all results for mpQPs directly translate to \eqref{eq:mpLDP}; for example, that the solution is piecewise affine over a polyhedral partition. This is also evident from the simple affine relationship between $x^*$ and $u^*$ in~\eqref{eq:sol-transform}. 
\end{remark}

\subsection{The explicit solution}
Necessary and sufficient conditions for a solution $u^*$ to the mpLDP in \eqref{eq:mpLDP} are the KKT-conditions
\begin{subequations}
    \label{eq:kkt}
  \begin{align}
      u^* + M^T \lambda &= 0, \\ 
      M u^* &\leq d(\theta), \\
      \lambda &\geq 0, \\
      [d(\theta)-M u^*]_i [\lambda]_i &= 0, \quad \forall i \in \mathbb{Z}_{1:m} \label{eq:kkt-compslack},
  \end{align}
\end{subequations}
with the dual variable $\lambda \in \mathbb{R}^m$. Both $u^*(\theta)$ and $\lambda(\theta)$ are functions of the parameter $\theta$, although we will often skip writing out this parameter dependence explicitly and use the notation $u^*$ and $\lambda$.

The main complication with using the KKT-conditions in \eqref{eq:kkt} to find a solution is the complementary slackness condition in \eqref{eq:kkt-compslack}, which make solving \eqref{eq:mpLDP} a combinatorial problem. To make the combinatorial aspect of solving \eqref{eq:mpLDP} more explicit, we introduce the notion of an active set.  

\begin{definition}[Active set]
    An index set $\mathcal{A} \subseteq \mathbb{Z}_{1:m}$ is an \textit{active set} to the mpLDP in \eqref{eq:mpLDP} if the equality constraints $M_i u = d_i(\theta)$ for all $i\in \mathcal{A}$ and $\lambda_j = 0$ for all $i\notin \mathcal{A}$ are imposed in the KKT conditions in \eqref{eq:kkt}.
\end{definition}

In other words, an active set forces all constraints in it to hold with equality (inequality constraints that holds with equality are said to be active, hence the name active set.) 

For a given active set $\mathcal{A}$, the KKT conditions reads  
\begin{subequations}
    \label{eq:kkt-AS}
  \begin{align}
      u^* + \sum_{i \in \mathcal{A}} M^T_i \lambda_i  = 0, \label{eq:stat-AS}&& \\ 
      M_{{\mathcal{A}}} u^* = d_{\mathcal{A}}(\theta),\quad \:\: \lambda_{\mathcal{A}} \geq 0, \label{eq:feas-AS} && \\
      M_{\bar{\mathcal{A}}} u^* \leq d_{\bar{\mathcal{A}}}(\theta), \quad \:\: \lambda_{\bar{\mathcal{A}}} = 0,&&
  \end{align}
\end{subequations}
which, in contrast to \eqref{eq:kkt}, is a system of linear equality and inequality constraints.

To form an explicit solution to \eqref{eq:mpLDP}, we are interested in parameters for which a given active set $\mathcal{A}$ leads to a solvable system \eqref{eq:kkt-AS}. The set of all such parameters for a given active set is known as a \textit{critical region}:
\begin{definition}[Critical region]
    The critical region $\Theta_{\mathcal{A}}$ for a given active set $\mathcal{A}$ to \eqref{eq:mpLDP} is defined as the set 
    \begin{equation*}
        \label{eq:CR-implicit}
    \Theta_{\mathcal{A}} \triangleq \{\theta\in \Theta_0 : \exists (u^*,\lambda) \in \mathbb{R}^n \times \mathbb{R}^m \text{ that satisfy \eqref{eq:kkt-AS}}\}.
    \end{equation*}
\end{definition}

This definition of a critical region is implicit. If we assume that the matrix $M_{\mathcal{A}}$ has full row rank, formalized below, it is possible to give an explicit expression of $\Theta_{\mathcal{A}}$.

\begin{definition}[LICQ]
    The \textit{linear independence constraint qualification} (LICQ) is satisfied for an active set $\mathcal{A}$ if the matrix $M_{\mathcal{A}}$ has full row rank (i.e., if the rows of $M_{\mathcal{A}}$ are linearly independent.)
\end{definition}

If LICQ holds for $\mathcal{A}$, an explicit expression of the critical region exists. To see this, first note that if LICQ holds for $\mathcal{A}$ we get, by combining \eqref{eq:stat-AS} and \eqref{eq:feas-AS}, that the dual variable can be uniquely determined by 
\begin{equation}
    \label{eq:lambda}
    \lambda_{\mathcal{A}}(\theta) = - \left(M_{\mathcal{A}} M_{\mathcal{A}}^T\right)^{-1} d_{\mathcal{A}}(\theta).
\end{equation}
Moreover, \eqref{eq:stat-AS} then directly gives the optimal primal variable $u^*(\theta)$ as 
\begin{equation}
    \label{eq:u}
  u^*(\theta) = - M^T_{\mathcal{A}} \lambda_{\mathcal{A}}(\theta),
\end{equation}
and the corresponding primal slack $\mu_{\bar{\mathcal{A}}}(\theta)$ for the inactive constraints $\bar{\mathcal{A}}$ is 
\begin{equation}
    \label{eq:mu}
    \mu_{\bar{\mathcal{A}}}(\theta) = [d(\theta)]_{\bar{\mathcal{A}}}-[M]_{\bar{\mathcal{A}}} u^*(\theta).
\end{equation}

Since $d(\theta)$ is affine in $\theta$, we have that $\lambda_{\mathcal{A}}(\theta)$, $u^*(\theta)$ and $\mu_{\bar{\mathcal{A}}}(\theta)$ are also affine in $\theta$ and, hence, the critical region $\Theta_{\mathcal{A}}$ for the active set $\mathcal{A}$ is the polyhedron
\begin{equation}
  \label{eq:cA}
  \Theta_{\mathcal{A}}\triangleq\{\theta\in \Theta_0 : \mu_{\bar{\mathcal{A}}}(\theta) \geq 0, \lambda_{\mathcal{A}}(\theta) \geq 0 \}.
\end{equation}

Consequently, the explicit solution $u^*(\theta)$ to \eqref{eq:mpLDP} is the polyhedral piecewise-affine function
\begin{equation}
    \label{eq:exp-sol}
    u^*(\theta) = -M_{\mathcal{A}}^T(M_{\mathcal{A}} M^T_{\mathcal{A}})^{-1} d_{\mathcal{A}}(\theta),\quad \forall \theta \in \Theta_{\mathcal{A}}.
\end{equation}
\begin{remark}[Explicit solution to mpQP]
    Note that \eqref{eq:exp-sol} in combination with \eqref{eq:sol-transform} directly gives $x^*(\theta)$, and that the critical regions $\Theta_{\mathcal{A}}$ are \textit{the same}.
\end{remark}

The main challenge for determining the explicit solution in \eqref{eq:exp-sol} is to find the active sets that define the critical regions.
Formally, this corresponds to finding the set $\mathbb{A} \triangleq \{\mathcal{A} : \Theta_{\mathcal{A}} \neq \emptyset\}$. From a practical point of view, however, expressing the explicit solution only requires active sets that define critical regions that cover the parameter space. Therefore, active sets that break the LICQ can be discarded (see, for example, Lemma 1 in \cite{ahmadi2018combinatorial}, which ensures that active sets that break LICQ can be discarded, even if they define full-dimensional critical regions). In summary, finding the explicit solution \eqref{eq:mpLDP} can be formalized as 
\begin{tcolorbox}
\begin{problem}
    \label{prob:main}
    Find the set $\mathbb{A}^* = \mathbb{A}^{\text{LICQ}}$, where 
    \begin{equation*}
        \mathbb{A}^{\text{LICQ}} \triangleq \{\mathcal{A} : \Theta_{\mathcal{A}} \neq \emptyset \text{ and } \mathcal{A} \text{ satisfies LICQ}\}.
    \end{equation*}
\end{problem}
\end{tcolorbox}

\subsection{Existing methods to compute the explicit solution}
\label{ssec:review-methods}
Traditionally, Problem \ref{prob:main} has been tackled with \textit{geometrical} methods. These methods start in a critical region and explore all neighboring regions by moving in the parameter space $\mathbb{R}^{p}$ \cite{bemporad2002explicit,baotic2002efficient,tondel2003algorithm}. The most efficient geometrical methods exploit the ``facet-to-facet'' property \cite{spjotvold2006facet}, which  allow neighboring regions to be accessed from the facets of the current critical region. A major challenge for geometrical methods is that this ``facet-to-facet'' property does not always hold \cite{spjotvold2006facet}. Another challenge is that they employ geometrical operations such as computing points on lower-dimensional facets, which is a numerically unreliable operation \cite{herceg2015enumeration}. Therefore, geometrical methods themselves are often unreliable, especially when the dimension of the parameter space increases.

In contrast, \textit{combinatorial} methods do not explore the parameter space, but instead search directly for active sets that leads to non-empty critical regions. A naive combinatorial method would be to solve feasibility problems to see if $\Theta_{\mathcal{A}}\neq \emptyset$ \textit{for all possible} $2^m$ active sets. Since this would require an exponential number of feasibility problems to be solved, this quickly becomes intractable as the number of constraint $m$ increases. Instead of such a brute-force search, combinatorial methods use properties of the problem to dismiss $\mathcal{A}$ that cannot possibly be optimal based on previously tested active sets (known as \textit{fathoming}). The first work in this directions was \cite{gupta2011novel}, which used fathoming based on primal feasibility and LICQ violation to prune candidate active sets. Following this work, several works (e.g. \cite{feller2013improved, oberdieck2017explicit,ahmadi2018combinatorial,herceg2015enumeration})  have introduced additional fathoming strategies, which ultimately requires fewer feasibility problems to be solved. 

\section{A combinatorial connected-graph algorithm}
\label{sec:algorithm}
In this section, we propose a combinatorial method that solves Problem \ref{prob:main}; that is, a combinatorial method that computes the explicit solution to the mpLDP in \eqref{eq:mpLDP}. First, we introduce the concepts of active sets being \textit{geometrically} and \textit{combinatorially} adjacent. We then use these concepts to construct a simple algorithm that produces a solution $\mathbb{A}^*$ to Problem~1. Finally, we discuss the method's relationship with the connected-graph methods in \cite{herceg2015enumeration,ahmadi2018combinatorial}. The foundation for the proposed method is the concept of combinatorially adjacent active sets, introduced next, which complement the classical concept of geometrical adjacency of critical regions.

\subsection{Geometrical and combinatorial adjacency}

As mentioned in Section \ref{ssec:review-methods}, both geometrical methods \cite{bemporad2002explicit,tondel2003algorithm,grieder2004computation,jones2006multiparametric} and the connected-graph method in \cite{herceg2015enumeration} are based on the fact that critical regions are adjacent with other critical regions in the parameter space. Two critical regions are adjacent if they intersect, and we say that the corresponding active sets are geometrically adjacent.
\begin{definition}[Geometrical adjacency]
    Two active sets $\mathcal{A},\tilde{\mathcal{A}}\in \mathbb{A}$ are \textit{geometrically adjacent} if $\Theta_{\mathcal{A}} \cap \Theta_{\tilde{\mathcal{A}}} \neq \emptyset$. 
\end{definition}

Geometrical methods use this kind of adjacency to find the explicit solution by ``jumping'' between adjacent regions. Specifically, the methods compute the boundary (facets) of a critical region and then determine adjacent critical regions based on them. As is pointed out in \cite{herceg2015enumeration}, geometrical operations on facets are often numerically unstable.
Our goal is therefore to avoid any geometrical operations and instead consider adjacency purely in terms of the active sets. Intuitively, two active sets are adjacent if one of the sets can be transformed into the other by either adding or removing a single constraint (put in another way: that the \textit{Hamming distance} between the active sets is $1$). We define this type of adjacency as \textit{combinatorial} adjacency. 
\begin{definition}[Combinatorial adjacency]
    \label{def:comb-adj}
    Two active sets $\mathcal{A},\tilde{\mathcal{A}}\in \mathbb{A}$ are \textit{combinatorially adjacent} if $\mathcal{A} = \tilde{\mathcal{A}}\cup \{i\}$ or $\tilde{\mathcal{A}} = \mathcal{A}\cup \{i\}$ for some $i \in \mathbb{Z}_{1:m}$. 
\end{definition}

In the proposed algorithm, soon to be presented, candidate active sets are explored by jumping between combinatorially adjacent active sets. Specifically, such exploration will be done through a particular class of sequences of combinatorially adjacent active sets that we call  \textit{valid combinatorial sequences}.
\begin{definition}[Valid combinatorial sequence]
    \label{def:vcs}
    A sequence $\{\mathcal{A}_i\}_i^{N}$ with $\mathcal{A}_i \in \mathbb{A}$ is a \textit{valid combinatorial sequence} if for all $i = 1, \dots, N-1$%
    \begin{itemize}
        \item[(i)] $\mathcal{A}_i$ and $\mathcal{A}_{i+1}$ are combinatorially adjacent. 
        \item[(ii)] $\mathcal{A}_i \in \mathbb{A}^{\text{LICQ}}$ or $\mathcal{A}_{i+1} \in \mathbb{A}^{\text{LICQ}}$.
    \end{itemize}
\end{definition}

Point (i) in Definition \ref{def:vcs} simply states that the sequence should be ``connected'', while point (ii) makes sure that there is not a chain of degenerate active sets in the sequence. Point (ii) further implies that if LICQ breaks for a set in the sequence, only its subsets, and not supersets, can be the next set in a valid sequence. This will be exploited in the proposed method (specifically at Step \ref{step:licq-candidates} of Algorithm \ref{alg:updown}). 

We use valid combinatorial sequences to define \textit{combinatorial connectedness} between active sets in $\mathbb{A}$, which will use in Section \ref{ssec:algorithm} to explore active set candidates.

\begin{definition}[Combinatorially connected]
    \label{def:comb-connected}
    Two active sets $\mathcal{A}$ and $\tilde{\mathcal{A}}$ are \textit{combinatorially connected} if there exists a valid combinatorial sequence $\{\mathcal{A}_i\}_{i=1}^N$ with $\mathcal{A}_i \in \mathbb{A}$ such that $\mathcal{A}_1=\mathcal{A}$ and $\mathcal{A}_N=\tilde{\mathcal{A}}$. 
\end{definition}

Next, we show that there is a relationship between two active sets being geometrically adjacent and combinatorially connected.
\begin{lemma}[geometrical $\to$ combinatorial]
    \label{lem:geocomb}
    If two active sets $\mathcal{A},\tilde{\mathcal{A}} \in \mathbb{A}^{\text{LICQ}}$ are geometrically adjacent, they are combinatorially connected.
\end{lemma}
\begin{proof}
    See Appendix.
\end{proof}

\begin{remark}[adjacency vs connectedness]
    Note that two active sets being geometrically adjacent does \textit{not} generally imply that they are combinatorially \textit{adjacent}, but it does imply that they are combinatorially \textit{connected}. An illustrative example of this distinction is given in Example~1 in \cite{spjotvold2006facet}. Geometrical adjacency does, however, imply combinatorial adjacency when no degeneracies occur, which follows directly from Theorem 2 in \cite{tondel2003algorithm}.
\end{remark}

Now we are ready to present the main theoretical results of this paper: namely, that any two active set that are optimal and satisfy LICQ are combinatorially connected. This theorem is the foundation of the correctness of the proposed method.
\begin{tcolorbox}
\begin{theorem}[Combinatorially connected graph]
    \label{th:main}
    Any pair $\mathcal{A},\tilde{\mathcal{A}} \in \mathbb{A}^{\text{LICQ}}$ is combinatorially connected.
\end{theorem}
\end{tcolorbox}
\begin{proof}
    It is well-known that critical regions form a connected graph in the geometrical sense \cite{baotic2002efficient}, which follows directly from there being no ``holes'' in the partition that defines the explicit solution. Together with Lemma \ref{lem:geocomb}, this implies that the active sets that define the critical regions also form a connected graph in a combinatorial sense. 
\end{proof}

\begin{remark}[Theorem~\ref{th:main} and \cite{oberdieck2017explicit}]
    \label{rem:connected}
    A similar result to Theorem \ref{th:main} is presented in \cite{oberdieck2017explicit}, but their result is, as is pointed out in \cite{ahmadi2019explicit}, based on an incorrect premise that weakly active constraints cannot occur. Moreover, the result therein is more loosely proved in terms of ``dual simplex steps'', which is not as direct as the concept of combinatorial adjacency introduced in Definition \ref{def:comb-adj} that Theorem \ref{th:main} is based on.
\end{remark}
\begin{remark}[Lower-dimensional regions]
    Importantly, note that the critical region $\Theta_{\mathcal{A}}$ for any $\mathcal{A} \in \mathbb{A}$ might be lower-dimensional. This is central for connectedness in degenerate cases. As is highlighted by, e.g., Example~1 in \cite{ahmadi2018combinatorial}, restricting the exploration to full-dimensional critical regions does not lead to the desired connectedness in Theorem~\ref{th:main}. 
\end{remark}

\subsection{A combinatorial algorithm}
\label{ssec:algorithm}
We are now interested in using the insights from Theorem~\ref{th:main} to construct an algorithm that solves Problem \ref{prob:main}. That is, we are interested in constructing an algorithm that produces a collection of active sets $\mathbb{A}^*$ that coincides with $\mathbb{A}^{\text{LICQ}}$. 
We will do this by recursively generating combinatorially adjacent active sets for all active sets that define critical regions that are non-empty.

The proposed algorithm, given in Algorithm \ref{alg:updown}, considers an unexplored active set $\mathcal{A}$ in each iteration. If LICQ holds for $\mathcal{A}$, the corresponding region of the form  \eqref{eq:cA} is formed and a feasibility problem is solved to check whether $\Theta_{\mathcal{A}} \neq \emptyset$. If $\Theta_\mathcal{A}$ is non-empty, we add $\mathcal{A}$ to the set of discovered optimal active sets $\mathbb{A}^*$, and put all its unexplored combinatorially adjacent active sets on the  stack $S$ for further exploration. All combinatorially adjacent active sets to $\mathcal{A}$ are formed with the functions \textsc{ExploreSupersets}, which creates candidate active sets by adding an index to $\mathcal{A}$, and with the function \textsc{ExploreSubsets}, which creates candidate active sets by removing an index from $\mathcal{A}$. By forming all unexplored active sets that are combinatorially adjacent to $\mathcal{A}$, we will ensure (based on Theorem \ref{th:main}) that all active sets in $\mathbb{A}^{\text{LICQ}}$ are explored. If LICQ does not hold for $\mathcal{A}$, we only form combinatorially adjacent active sets by removing a constraints, since LICQ is guaranteed to be broken for combinatorially adjacent active sets that have more elements. This is sufficient, since, for a valid combinatorial sequence, LICQ does not break consecutively (cf. point (ii) in Definition \ref{def:vcs}.)
\begin{algorithm}[H]
    \caption{Combinatorial method for solving Problem \ref{prob:main}.}
  \label{alg:updown}
  \begin{algorithmic}[1]
      \Require $\Theta_0\subseteq \mathbb{R}^p$, $\mathcal{A}_0 \in \mathbb{A}^{\text{LICQ}}$, an mpLDP of the form \eqref{eq:mpLDP}
	\Ensure Collection of optimal active sets $\mathbb{A}^*$ over $\Theta$ 
	\State $S \leftarrow \{\mathcal{A}_0\}$,\quad $\mathcal{E}\leftarrow \{\mathcal{A}_0\}$
	\While{$S\neq \emptyset$}
	\State Pop $\mathcal{A}$ from $S$
	\If {LICQ satisfied for $\mathcal{A}$} 
	\State Compute $\lambda_{\mathcal{A}}(\theta)$ accodring to \eqref{eq:lambda}
    \State Compute $\mu_{\bar{\mathcal{A}}}(\theta)$ according to \eqref{eq:u} and \eqref{eq:mu}
	\State Form $\Theta_{\mathcal{A}}$ according to \eqref{eq:cA}
	\If{$\Theta_{\mathcal{A}} \neq \emptyset$}
	\State Add $\mathcal{A}$ to $\mathbb{A}^*$
	\State \textsc{ExploreSupersets}($\mathcal{A},\mathcal{E},S$)
	\State \textsc{ExploreSubsets}($\mathcal{A},\mathcal{E},S$)
	\EndIf
    \Else\:\:\textsc{ExploreSubsets}($\mathcal{A},\mathcal{E},S$)\label{step:licq-candidates}
	\Comment{{\footnotesize($\mathcal{A}$ violates LICQ)}}
	\EndIf
	\EndWhile
	\hrule\rule{0pt}{1pt}
	\Procedure{ExploreSupersets}{$\mathcal{A},\mathcal{E},S$}
	\For{$i \in \bar{\mathcal{A}}$}
	\State $\mathcal{A}^+ \leftarrow \mathcal{A} \cup \{i\}$
	\If{$\mathcal{A}^+\notin \mathcal{E}$} add $\mathcal{A}^+$ to $\mathcal{E}$ and $S$
	\EndIf
	\EndFor
	\EndProcedure
    \vspace{2pt}
	\hrule\rule{0pt}{1pt}
	\Procedure{ExploreSubsets}{$\mathcal{A},\mathcal{E},S$}
	\For{$i \in {\mathcal{A}}$}
	\State $\mathcal{A}^- \leftarrow \mathcal{A} \setminus \{i\}$
	\If{$\mathcal{A}^-\notin \mathcal{E}$} add $\mathcal{A}^-$ to $\mathcal{E}$ and $S$
	\EndIf
	\EndFor
	\EndProcedure

  \end{algorithmic}
\end{algorithm}

Since Algorithm \ref{alg:updown} produces all combinatorially adjacent active sets for any $\mathcal{A}\in \mathbb{A}^{\text{LICQ}}$, Definition \ref{def:comb-connected} and Theorem 1 directly guarantees correctness, in the sense that Algorithm~\ref{alg:updown} solves Problem \ref{prob:main}, formalized in the following corollary. 
\begin{corollary}[Correctness of Algorithm \ref{alg:updown}] The output of Algorithm~\ref{alg:updown} is $\mathbb{A}^* = \mathbb{A}^{\text{LICQ}}$. 
\end{corollary}
\begin{remark}[Finding $\mathcal{A}_0$]
    To find an active set $\mathcal{A}_0 \in \mathbb{A}^{\text{LICQ}}$ to initialize Algorithm \ref{alg:updown} with, one can use a QP solver, such as \texttt{DAQP} \cite{arnstrom2022daqp}, and solve \eqref{eq:rmpc-mpqp} for a given parameter $\theta \in \Theta_0$. Another possibility is to initially perform a combinatorial exploration akin to the method in \cite{gupta2011novel}. 
\end{remark}

\subsection{Comparison with similar methods}
\label{ssec:compare}

As previously mentioned, the proposed method given in Algorithm \ref{alg:updown} is related to the methods in \cite{oberdieck2017explicit} and \cite{ahmadi2018combinatorial}. To highlight the contributions of this paper, we will now delineate some important differences between \cite{oberdieck2017explicit,ahmadi2018combinatorial} and Algorithm~\ref{alg:updown}.

The connected-graph approach presented in \cite{oberdieck2017explicit} use Theorem 2 in \cite{tondel2003algorithm} to characterize the facets of a critical regions and use this to generate new active set candidates. As previously mentioned, geometrical operations that involve facets are often numerically unstable. Moreover, to obtain the facets, redundant half-planes need to be removed from the critical regions, which lead to the main computational burden (as is emphasized in \cite{tondel2003algorithm} and reported in Figure 6 of \cite{oberdieck2017explicit}). In contrast, Algorithm~\ref{alg:updown} does not need to characterize the facets of each critical region, making it more numerically robust and efficient (supported by the results in Section \ref{sec:result}.) Moreover, the proof of the correctness of the method in \cite{oberdieck2017explicit} is based on false premises in degenerate cases, see Remark~\ref{rem:connected}, while Algorithm~\ref{alg:updown} is based on Theorem \ref{th:main} that still holds for degenerate problems.

In \cite{ahmadi2018combinatorial}, degeneracies are handled by detecting constraints that are weakly active/inactive. This requires feasibility problems of the form  
\begin{equation}
    \label{eq:infeasibility-lifted}
  \begin{aligned}
	&\underset{t,\theta\in\Theta_0, \lambda, \mu}{\text{minimize}}&& t \\
    &\text{subject to} && M_{\mathcal{A}}M_{\mathcal{A}}^T \lambda = d_{\mathcal{A}}(\theta),\qquad \qquad \lambda \geq t,  \\
    & && \mu =d_{\bar{\mathcal{A}}}(\theta) + M_{\bar{\mathcal{A}}} M_{\mathcal{A}}^T \lambda,\qquad \:\mu \geq t. 
  \end{aligned}
\end{equation}
to be solved, where $t=0$ signifies a degenerate case with weakly active/inactive constraints. 
Instead of identifying weakly inactive/active constraints as in \cite{ahmadi2018combinatorial}, the proposed method exploit that there always exist a sequence of active sets corresponding to critical regions (possibly lower-dimensional) that are non-empty that connect any two full-dimensional critical regions. The existence of such lower-dimensional critical regions is hinted at in Example 1 in \cite{ahmadi2018combinatorial}, but is never proved nor exploited therein.   

Since Algorithm~\ref{alg:updown} does not have to detect weakly active/inactive constraints, only simple LDPs of the form    
\begin{equation}
    \label{eq:infeasibility-reduced}
    \min_{\theta\in \Theta_{\mathcal{A}}} \|\theta\|^2_2
\end{equation}
need to solved in the proposed algorithm. As a result, the dual active-set solver \texttt{DAQP} \cite{arnstrom2022daqp} can be used to efficiently check feasibility for $\Theta_{\mathcal{A}}$. As is illustrated in the results in Section \ref{sec:result}, this leads to a significant speedup.

\begin{remark}[Dimension of feasibility problems]
    The feasibility problems that need to be solved in \cite{ahmadi2018combinatorial} of the form \eqref{eq:infeasibility-lifted} are carried out over $1+p+2m$ dimensions ($t,\theta,\lambda,\mu$). In contrast, the feasibility problems that need to be solved in Algorithm~\ref{alg:updown} of the form \eqref{eq:infeasibility-reduced} are carried out over $p$ dimensions ($\theta$).
\end{remark}

Another advantage of the straightforward degeneracy handling in Algorithm~\ref{alg:updown} is that the algorithm itself is a lot simpler than the proposed algorithm in \cite{ahmadi2018combinatorial}. This makes it easier to implement efficiently.

\section{Numerical Experiments}
\label{sec:result}
To illustrate the efficacy of Algorithm~\ref{alg:updown}, we compare a Julia implementation of it\footnote{\url{https://github.com/darnstrom/ParametricDAQP.jl}} with the state-of-the-art mpQP solvers in \texttt{MPT} \cite{herceg2013mpt} and \texttt{POP} \cite{oberdieck2016pop}. For \texttt{MPT}, we use its implementation of the geometrical method presented in \cite{jones2006multiparametric}. For \texttt{POP}, we use its implementation of the connected-graph method in \cite{oberdieck2017explicit}. Experiments\footnote{\url{https://github.com/darnstrom/cdc24-mpqp}} were carried out on the test set provided in \cite{oberdieck2016pop}, which consists of 100 mpQPs, for which the solvers were tasked to compute the explicit solution; for more information about the test set, see \cite{oberdieck2016pop}. To limit the total execution time for the entire problem set, we terminate a solver after 100 seconds if it has been unable to return a solution up until then. 

The results are reported in Figure \ref{fig:sol-curve}, where the implementation of Algorithm \ref{alg:updown} displays a speedup of about two orders of  magnitude compared with both \texttt{MPT} and \texttt{POP}. For a fair comparison, we tried several internal LP solvers in both \texttt{MPT} and \texttt{POP}. For \texttt{MPT}, the open-source solver \texttt{GLPK} gave the best performance. For \texttt{POP}, the proprietary solver \texttt{CPLEX} gave the best performance. As presented in Section~\ref{ssec:compare}, the proposed method use the open-source solver \texttt{DAQP} \cite{arnstrom2022daqp} for solving feasibility problems of the form \eqref{eq:infeasibility-reduced}.
Note that the execution times for the proposed method includes forming and storing the explicit solution for each active set in $\mathbb{A}^*$. 

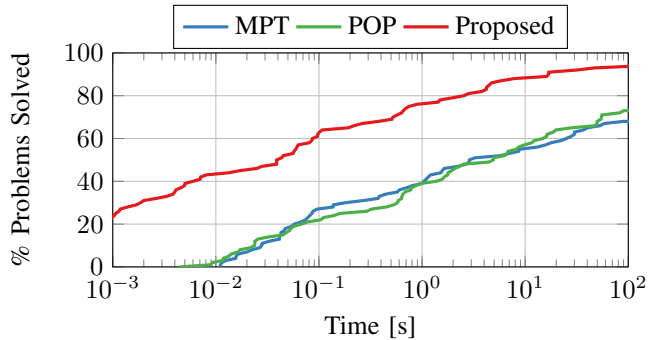
\begin{figure}
    \centering
    \begin{tikzpicture}
    \begin{axis}[
        xmode = log,
        xmin = 0.001,xmax = 100,
        ymin=0,ymax = 100,
        xlabel={Time [s]},
        ylabel={\% Problems Solved},
		legend style={at ={(0.5,1.225)},anchor=north}, ymajorgrids,xmajorgrids,
		legend cell align={left},legend columns=3,
		x post scale=1,
		y post scale=0.5,
        ]
        \addplot [set19c2,very thick] table [y expr=\coordindex, x={time}] {\mptglpk}; 
        \addplot [set19c3,very thick] table [y expr=\coordindex, x={time}] {\POP}; 
        \addplot [set19c1,very thick] table [y expr=\coordindex, x={time}] {\cholrank}; 
        \legend{ MPT, POP, Proposed} 
    \end{axis}
\end{tikzpicture}
    \caption{\small Time taken for \texttt{MPT} \cite{herceg2013mpt}, \texttt{POP} \cite{oberdieck2017explicit}, and a Julia implementation of Algorithm \ref{alg:updown} (``Proposed''), to compute the explicit solution for the benchmark mpQP problems from the \texttt{POP} toolbox \cite{oberdieck2016pop}.}
    \label{fig:sol-curve}
\end{figure}

\section{Conclusion}
We have proposed a combinatorial method for computing explicit solutions to multi-parametric quadratic programs. The method builds on optimal active sets being ``combinatorially connected'', which makes the explicit solution form a combinatorially connected graph. We show that an implementation of the proposed method can yield a speedup of two orders of magnitude compared to state-of-the-art software packages such as \texttt{MPT} and \texttt{POP}. 

{Future work include presenting details of how to implement Algorithm \ref{alg:updown} efficiently, and to develop a parallelized version of it.}

\bibliographystyle{IEEEtran}
\linespread{1.0}\selectfont
\bibliography{lib.bib}

\begin{thebibliography}{10}
\providecommand{\url}[1]{#1}
\csname url@samestyle\endcsname
\providecommand{\newblock}{\relax}
\providecommand{\bibinfo}[2]{#2}
\providecommand{\BIBentrySTDinterwordspacing}{\spaceskip=0pt\relax}
\providecommand{\BIBentryALTinterwordstretchfactor}{4}
\providecommand{\BIBentryALTinterwordspacing}{\spaceskip=\fontdimen2\font plus
\BIBentryALTinterwordstretchfactor\fontdimen3\font minus
  \fontdimen4\font\relax}
\providecommand{\BIBforeignlanguage}[2]{{%
\expandafter\ifx\csname l@#1\endcsname\relax
\typeout{** WARNING: IEEEtran.bst: No hyphenation pattern has been}%
\typeout{** loaded for the language `#1'. Using the pattern for}%
\typeout{** the default language instead.}%
\else
\language=\csname l@#1\endcsname
\fi
#2}}
\providecommand{\BIBdecl}{\relax}
\BIBdecl

\bibitem{rawlings2017model}
J.~B. Rawlings, D.~Q. Mayne, M.~Diehl \emph{et~al.}, \emph{Model predictive
  control: theory, computation, and design}.\hskip 1em plus 0.5em minus
  0.4em\relax Nob Hill Publishing Madison, WI, 2017, vol.~2.

\bibitem{borrelli2017predictive}
F.~Borrelli, A.~Bemporad, and M.~Morari, \emph{Predictive control for linear
  and hybrid systems}.\hskip 1em plus 0.5em minus 0.4em\relax Cambridge
  University Press, 2017.

\bibitem{bemporad2002explicit}
A.~Bemporad, M.~Morari, V.~Dua, and E.~N. Pistikopoulos, ``The explicit linear
  quadratic regulator for constrained systems,'' \emph{Automatica}, vol.~38,
  no.~1, pp. 3--20, 2002.

\bibitem{tondel2003algorithm}
P.~T{\o}ndel, T.~A. Johansen, and A.~Bemporad, ``An algorithm for
  multi-parametric quadratic programming and explicit {MPC} solutions,''
  \emph{Automatica}, vol.~39, no.~3, pp. 489--497, 2003.

\bibitem{grieder2004computation}
P.~Grieder, F.~Borrelli, F.~Torrisi, and M.~Morari, ``Computation of the
  constrained infinite time linear quadratic regulator,'' \emph{Automatica},
  vol.~40, no.~4, pp. 701--708, 2004.

\bibitem{jones2006multiparametric}
C.~N. Jones and M.~Morrari, ``Multiparametric linear complementarity
  problems,'' in \emph{IEEE 45th Conference on Decision and Control
  (CDC)}.\hskip 1em plus 0.5em minus 0.4em\relax IEEE, 2006, pp. 5687--5692.

\bibitem{gupta2011novel}
A.~Gupta, S.~Bhartiya, and P.~Nataraj, ``A novel approach to multiparametric
  quadratic programming,'' \emph{Automatica}, vol.~47, no.~9, pp. 2112--2117,
  2011.

\bibitem{feller2013improved}
C.~Feller, T.~A. Johansen, and S.~Olaru, ``An improved algorithm for
  combinatorial multi-parametric quadratic programming,'' \emph{Automatica},
  vol.~49, no.~5, pp. 1370--1376, 2013.

\bibitem{oberdieck2017explicit}
R.~Oberdieck, N.~A. Diangelakis, and E.~N. Pistikopoulos, ``Explicit model
  predictive control: A connected-graph approach,'' \emph{Automatica}, vol.~76,
  pp. 103--112, 2017.

\bibitem{ahmadi2018combinatorial}
P.~Ahmadi-Moshkenani, T.~A. Johansen, and S.~Olaru, ``Combinatorial approach
  toward multiparametric quadratic programming based on characterizing adjacent
  critical regions,'' \emph{IEEE Transactions on Automatic Control}, vol.~63,
  no.~10, pp. 3221--3231, 2018.

\bibitem{herceg2015enumeration}
M.~Herceg, C.~N. Jones, M.~Kvasnica, and M.~Morari, ``Enumeration-based
  approach to solving parametric linear complementarity problems,''
  \emph{Automatica}, vol.~62, pp. 243--248, 2015.

\bibitem{mitze2020dynamic}
R.~Mitze and M.~M{\"o}nnigmann, ``A dynamic programming approach to solving
  constrained linear--quadratic optimal control problems,'' \emph{Automatica},
  vol. 120, p. 109132, 2020.

\bibitem{herceg2013mpt}
M.~Herceg, M.~Kvasnica, C.~N. Jones, and M.~Morari, ``Multi-parametric toolbox
  3.0,'' in \emph{2013 European Control Conference (ECC)}, 2013, pp. 502--510.

\bibitem{oberdieck2016pop}
R.~Oberdieck, N.~A. Diangelakis, M.~M. Papathanasiou, I.~Nascu, and E.~N.
  Pistikopoulos, ``Pop--parametric optimization toolbox,'' \emph{Industrial \&
  Engineering Chemistry Research}, vol.~55, no.~33, pp. 8979--8991, 2016.

\bibitem{hyb2004bemporad}
A.~Bemporad, ``{Hybrid Toolbox - User's Guide},'' 2004,
  \url{http://cse.lab.imtlucca.it/~bemporad/hybrid/toolbox}.

\bibitem{arnstrom2022unifying}
D.~Arnstr{\"o}m and D.~Axehill, ``A unifying complexity certification framework
  for active-set methods for convex quadratic programming,'' \emph{IEEE
  Transactions on Automatic Control}, vol.~67, no.~6, pp. 2758--2770, 2022.

\bibitem{baotic2002efficient}
M.~Baoti{\'c}, ``An efficient algorithm for multiparametric quadratic
  programming.''

\bibitem{spjotvold2006facet}
J.~Spj{\o}tvold, E.~C. Kerrigan, C.~N. Jones, P.~T{\o}ndel, and T.~A. Johansen,
  ``On the facet-to-facet property of solutions to convex parametric quadratic
  programs,'' \emph{Automatica}, vol.~42, no.~12, pp. 2209--2214, 2006.

\bibitem{ahmadi2019explicit}
P.~A. Moshkenani, ``Explicit model predictive control for higher order
  systems,'' Doctoral thesis, NTNU, 2019.

\bibitem{arnstrom2022daqp}
D.~Arnström, A.~Bemporad, and D.~Axehill, ``A dual active-set solver for
  embedded quadratic programming using recursive {LDL}$^{T}$ updates,''
  \emph{IEEE Transactions on Automatic Control}, vol.~67, no.~8, pp.
  4362--4369, 2022.

\bibitem{nocedal}
J.~Nocedal and S.~Wright, \emph{\BIBforeignlanguage{en}{Numerical
  Optimization}}.\hskip 1em plus 0.5em minus 0.4em\relax Springer Science \&
  Business Media, 2006.

\end{thebibliography}

\appendix

\begin{lemma}
    \label{lem:add}
    Assume that $\mathcal{A}$ and $\tilde{\mathcal{A}}$ are geometrically adjacent. Then for any $k \in \tilde{\mathcal{A}}$ the set $\Theta_{\mathcal{A} \cup \{k\}} \neq \emptyset$. 
\end{lemma}
\begin{proof}
    Since $\Theta_\mathcal{A}\neq \emptyset$ we have that the KKT-conditions in \eqref{eq:kkt-AS} are satisfied for $\mathcal{A}$. Then due to $\mathcal{A}$ and $\tilde{\mathcal{A}}$ being geometrically adjacent, the KKT-conditions for $\mathcal{A}\cup \{k\}$ will also be satisfied with the same optimizer $u^*$ and multipliers on $\Theta_{\mathcal{A}} \cap \Theta_{\tilde{\mathcal{A}}}$, with the additional element $\lambda_k = 0$ for the dual variable. As a result $\Theta_{\mathcal{A} \cup \{k\}}\neq \emptyset$.
\end{proof}

\begin{lemma}
    \label{lem:remove}
Assume $\mathcal{A}$ and  $\tilde{\mathcal{A}}$ are geometrically adjacent. Moreover, assume that for some $i\in \tilde{\mathcal{A}}$ the set $\mathcal{A}\cup \{i\}$ violates LICQ. Then there exists $j \in \mathcal{A} \setminus \tilde{\mathcal{A}}$ such that $\mathcal{A}\cup\{i\}\setminus\{j\} \in \mathbb{A}^{\text{LICQ}}$. 
\end{lemma}

\begin{proof}
    We will prove the case when $|\mathcal{A}|=n$ in detail; if $|\mathcal{A}|< n$ the same arguments hold by viewing the situation in a lower-dimensional affine subspace (see the end of the proof for some details). Assuming that $|\mathcal{A}|=n$, the optimizer $u^*$ is constrained to a single point in $\mathbb{R}^n$, and $\mathcal{A},\tilde{\mathcal{A}} \in \mathbb{A}^{\text{LICQ}}$ being geometrically adjacent ensures that such an optimizer $u^*$ exists for some parameter in $\tilde{\theta}\in \Theta_{{\mathcal{A}}}\cap\Theta_{\tilde{\mathcal{A}}}$.  If $i\in\tilde{\mathcal{A}}$ is added to $\mathcal{A}$, the LICQ will break, yet from Lemma \ref{lem:add} we have that $\mathcal{A}\cup\{i\}$ is still optimal for $\tilde{\theta}$. Since LICQ breaks, there exists a subset of constraint normals in $\mathcal{A}$ that is linearly dependent to $M_i$. Moreover, this subset does not include any constraints in $\mathcal{A}\cap \tilde{\mathcal{A}}$, since $i\in\tilde{\mathcal{A}}\in \mathbb{A}^{\text{LICQ}}$ implies that $M_i$ cannot be linearly dependent to any  $\{M_j\}_{j\in \tilde{\mathcal{A}}\setminus\{i\}}$. Therefore there exists some $j \in \mathcal{A}\setminus \tilde{\mathcal{A}}$ which makes the set of vectors $\{M_k\}_{k\in{\mathcal{A}\cup\{i\}\setminus\{j\}}}$ linearly independent. Imposing the constraint in $\mathcal{A}\cup\{i\}\setminus\{j\}$ will restrict $u^*$ to the same point as for $\mathcal{A}$, since $|\mathcal{A}\cup\{i\}\setminus\{j\}|=n$. Since $u^*$ was optimal for $\tilde{\theta} \in \Theta_{\mathcal{A}}\cap\Theta_{\tilde{\mathcal{A}}}$ we concluded that $\mathcal{A}\cup\{i\}\setminus\{j\}\in\mathbb{A}^{\text{LICQ}}.$  

If $|\mathcal{A}|<n$, the same arguments can be applied but on the affine subspace $\{u\in \mathbb{R}^n : M_{\mathcal{A}} u = d_{\mathcal{A}}(\tilde{\theta})\}$ after reduction to a lower dimension through a QR decomposition of $M_{\mathcal{A}}$ (see, for exampel, \cite[\S 15.3]{nocedal}.)
\end{proof}

\subsection{Proof Lemma \ref{lem:geocomb}}
\begin{proof}
    Based on Lemma \ref{lem:add} and Lemma \ref{lem:remove}, Algorithm~\ref{alg:vcs} presented below will be executable and form a valid combinatorial sequence with endpoints $\mathcal{A}$ and $\tilde{\mathcal{A}}$.
    \begin{algorithm}[H]
      \caption{}
      \label{alg:vcs}
      \begin{algorithmic}[1]
          \Require $\mathcal{A}$, $\tilde{\mathcal{A}} \in \mathbb{A}^{\text{LICQ}}$
          \Ensure {\small A valid combinatorial sequence with end points $\mathcal{A}$, $\tilde{\mathcal{A}}$} 
          \State  $\mathcal{A}^{-} \leftarrow  \mathcal{A}\setminus \tilde{\mathcal{A}}$;\qquad $\mathcal{A}^{+} \leftarrow  \tilde{\mathcal{A}}\setminus \mathcal{A}$
          \State $\mathcal{A}_1 \leftarrow \mathcal{A}$; \quad $k \leftarrow 1$
          \While {$\mathcal{A}^+ \neq \emptyset$ and $\mathcal{A}^- \neq \emptyset$}
          \If{$\mathcal{A}_k$ satisfy LICQ and $\mathcal{A}^+ \neq \emptyset$}
          \State $i \leftarrow$ pop from $\mathcal{A}^+$
          \State $\mathcal{A}_{k+1} \leftarrow \mathcal{A}_k \cup \{i\} $  \label{step:add-update}
          \Else
          \State $i \leftarrow$ select $i\in \mathcal{A}^{-}$ such that $\mathcal{A}_k \setminus \{i\} \in \mathbb{A}^{\text{LICQ}}$ \label{step:select-i}
          \State $\mathcal{A}^- \leftarrow \mathcal{A}^- \setminus \{i\} $ 
          \State $\mathcal{A}_{k+1} \leftarrow \mathcal{A}_k \setminus \{i\} $ 
          \EndIf
          \State $k \leftarrow k+1$
          \EndWhile
          \State \textbf{return} $\{\mathcal{A}_j\}_{j=1}^k$
      \end{algorithmic}
    \end{algorithm}

    Concretely, Lemma \ref{lem:add} ensures that $\mathcal{A}_{k+1}$ in Step \ref{step:add-update} of Algorithm \ref{alg:vcs} is in $\mathbb{A}$, and Lemma \ref{lem:remove} ensures that there will always exist a valid $i$ in Step \ref{step:select-i} of Algorithm \ref{alg:vcs} such that $\mathcal{A}_{k+1}$ returns to $\mathbb{A}^{\text{LICQ}}$.
\end{proof}
\end{document}